\documentclass[12pt, reqno]{amsart}
\usepackage{amsmath, amsthm, amscd, amsfonts, amssymb, graphicx, color}
\usepackage[bookmarksnumbered, colorlinks, plainpages]{hyperref}
\hypersetup{colorlinks=true,linkcolor=red, anchorcolor=green, citecolor=cyan, urlcolor=red, filecolor=magenta, pdftoolbar=true}

\newtheorem{theorem}{Theorem}[section]
\newtheorem{lemma}[theorem]{Lemma}
\newtheorem{proposition}[theorem]{Proposition}
\newtheorem{corollary}[theorem]{Corollary}
\theoremstyle{definition}
\newtheorem{definition}[theorem]{Definition}
\newtheorem{example}[theorem]{Example}

\numberwithin{equation}{section}
\DeclareUnicodeCharacter{2217}{-}
\begin{document}
	
	\setcounter{page}{1}
	
	\title{$\ast$-Operator frame for $Hom_{\mathcal{A}}^{\ast}(\mathcal{X})$ }
	
	\author{Roumaissae Eljazzar$^{1}$ \MakeLowercase{and} Mohamed Rossafi$^{2*}$}
	
	\address{$^{1}$Department of Mathematics, Faculty Of Sciences, University of Ibn Tofail, Kenitra, Morocco}
	\email{\textcolor[rgb]{0.00,0.00,0.84}{roumaissae.eljazzar@uit.ac.ma}}
	
	\address{$^{2}$LaSMA Laboratory Department of Mathematics Faculty of Sciences, Dhar El Mahraz University Sidi Mohamed Ben Abdellah, B. P. 1796 Fes Atlas, Morocco}
	\email{\textcolor[rgb]{0.00,0.00,0.84}{rossafimohamed@gmail.com; mohamed.rossafi@usmba.ac.ma}}
	
	
	
	\date{$^{*}$Corresponding author}
	
	\setcounter{page}{1}
	
	\title[$\ast$-Operator frame for $Hom_{\mathcal{A}}^{\ast}(\mathcal{X})$]{$\ast$-Operator frame for $Hom_{\mathcal{A}}^{\ast}(\mathcal{X})$}
	
	\maketitle
	\begin{abstract}
		In this Work, We introduce the concept of $\ast$-operator frame, which is a generalization of $\ast$-frames in Hilbert pro-$C^{\ast}$-modules, and we establish some results, we also study the tensor product of $\ast$-operator frame for Hilbert pro-$C^{\ast}$-modules.
	\end{abstract}
	\section{Introduction}
	In 1952, Duffin and Schaeffer \cite{Duf} introduced the notion of frame in nonharmonic Fourier analysis. In 1986 the work of Duffin and Schaeffer was continued by Grossman and Meyer \cite{Gross}. After their works, the theory of frame was developed and has been popular.
	
	The notion of frame on  Hilbert space has already been successfully extended to pro-$C^{\ast}$-algebras and Hilbert modules. In 2008, Joita \cite{Joita} proposed frames of multipliers in Hilbert pro-$C^{\ast}$-modules and showed that many properties of frames in Hilbert $C^{\ast}$-modules are valid for frames of multipliers in Hilbert modules over pro-$C^{\ast}$-algebras.
	
	$\ast-$frames is a generalization of $\ast$-operator frame, it was introduced by Alijani and Dehghan \cite{ALija}, as a generalization
	of frames in Hilbert C*-modules. Naroei and Nazari \cite{Naro} investigated The $\ast$-frame of multipliers in Hilbert pro-$C^{\ast}$-algebras.
	
	The first purpose of this paper is to give the definition of $\ast$-operator frame in pro-$C^{\ast}$-modules and some properties.
	
	The seconde purpose is to  investigate the tensor product of Hilbert pro-$C^{\ast}$-modules, and to show that tensor product of $\ast$-operator frames for Hilbert pro-$C^{\ast}$-modules $\mathcal{X}$ and $\mathcal{Y}$, present $\ast$-operator frame for $\mathcal{X} \otimes \mathcal{Y}$
	
	In the following section, we give some definitions and basic properties of Hilbert $C^{\ast}$-modules
	
	\section{Preliminaries}
	The basic information about pro-$C^{\ast}$-algebras can be found in the works \cite{Frag,Frago,Fragou,Mallios,Inoue,Philip,Philips}.
	
	$C^{\ast}$-algebra whose topology is induced by a family of continuous $C^{\ast}$-seminorms instead of a $C^{\ast}$-norm is called pro-$C^{\ast}$-algebra.  Hilbert pro-$C^{\ast}$-modules are generalizations of Hilbert spaces by allowing the inner product to take values in a pro-$C^{\ast}$-algebra rather than in the field of complex numbers.
	
	Pro-$C^{\ast}$-algebra is defined as a complete Hausdorff complex topological $\ast$-algebra $\mathcal{A}$ whose topology is determined by its continuous $C^{\ast}$-seminorms in the sens that a net $\{a_{\alpha}\}$ converges to $0$ if 	and only if $p(a_{\alpha})$ converges to $0$ for all continuous $C^{\ast}$-seminorm $p$ on $\mathcal{A}$ (see \cite{Inoue,Lance,Philips}), and we have:
	\begin{enumerate}
		\item[1)] $p(ab) \leq p(a)p(b)$
		\item[2)] $p(a^{\ast}a)=p(a)^{2}$
	\end{enumerate}
	for all  $a , b \in \mathcal{A}$.
	
	If the topology of pro-$C^{*}$-algebra is determined by only countably many $C^{*}$-seminorms, then it is called
	a $\sigma$-$C^{*}$-algebra.
	
	We denote by $sp(a)$ the spectrum of $a$ such that: $sp(a)=\left\{\lambda \in \mathbb{C}: \lambda 1_{\mathcal{A}}-a\right.$ is not invertible $\}$ for all $a \in \mathcal{A}.$ Where $\mathcal{A}$ is unital pro-$C^{*}$-algebra with unite $1_{\mathcal{A}}$.
	
	The set of all continuous $C^{\ast}$-seminorms on $\mathcal{A}$ is denoted by $S(\mathcal{A})$.
	If $\mathcal{A}^{+}$ denotes the set of all positive elements of $\mathcal{A}$, then $\mathcal{A}^{+}$ is a closed convex $C^{*}$-seminorms on $\mathcal{A}.$ 
	\begin{example}
		Every $C^{*}$-algebra is a pro-$C^{*}$-algebra.
		\end{example}
	\begin{proposition} 
		Let $\mathcal{A}$ be a unital pro-$C^{*}$-algebra with an identity $1_{\mathcal{A}}.$ Then for any $p \in S(\mathcal{A}),$ we have:
		\begin{enumerate}
			\item[(1)] $p(a)= p(a^{*})$ for all $a \in A$
			\item[(2)] $p \left(1_{\mathcal{A}}\right)=1$
			\item[(3)] If $a, b \in \mathcal{A}^{+}$ and $a \leq b$, then $ p(a) \leq p(b)$
			\item[(4)] If $1_{\mathcal{A}} \leq b$, then $b$ is invertible and $b^{-1} \leq 1_{\mathcal{A}}$
			\item[(5)] If $a, b \in \mathcal{A}^{+}$ are invertible and $0 \leq a \leq b$, then $0 \leq b^{-1} \leq a^{-1}$
			\item[(6)] If $a, b, c \in \mathcal{A}$ and $a \leq b$ then $c^{*} a c \leq c^{*} b c$
			\item[(7)] If $a, b \in \mathcal{A}^{+}$ and $a^{2} \leq b^{2}$, then $0 \leq a \leq b$
		\end{enumerate}
	\end{proposition}
	
	\begin{definition}\cite{Philips}
		A pre-Hilbert module over pro-$C^{\ast}$-algebra $\mathcal{A}$, is a complex vector space $E$ which is also
		a left $\mathcal{A}$-module compatible with the complex algebra structure, equipped with an $\mathcal{A}$-valued inner product $\langle .,.\rangle$ $E \times E \rightarrow \mathcal{A}$ which is $\mathbb{C}$-and $\mathcal{A}$-linear in its first variable and satisfies the following conditions:
		\begin{enumerate}
			\item[1)]
			$\langle \xi, \eta\rangle^{*}=\langle \eta, \xi\rangle 
			$ for every $\xi,\eta \in E$
			\item[2)]$\langle \xi, \xi\rangle \geq 0$ for every $\xi \in E$
			\item[3)] $\langle \xi, \xi\rangle=0 $  if and only if  $\xi=0$
		\end{enumerate}
		for every $\xi, \eta \in E .$ We say $E$ is a Hilbert $\mathcal{A}$-module (or Hilbert pro-$C^{\ast}$-module over $\mathcal{A}$ ). If $E$ is complete with respect to the topology determined by the family of seminorms
		$$
		\bar{p}_{E}(\xi)=\sqrt{p(\langle \xi, \xi\rangle)} \quad \xi \in E, p \in S(\mathcal{A})
		$$
	\end{definition}
	Let $\mathcal{A}$ be a pro-$C^{\ast}$-algebra and let $\mathcal{X}$ and $\mathcal{Y}$ be Hilbert $\mathcal{A}$-modules and assume that I and J be countable index sets.
	A bounded $\mathcal{A}$-module map from
	$\mathcal{X}$ to $\mathcal{Y}$ is called an operators from $\mathcal{X}$ to $\mathcal{Y}$. We denote the set of all operator from $\mathcal{X}$ to $\mathcal{Y}$ by  $Hom_{\mathcal{A}}(\mathcal{X}, \mathcal{Y})$.
	
	\begin{definition}
		An $ \mathcal{A}$-module map $T: \mathcal{X} \longrightarrow \mathcal{Y}$  is adjointable if there is a map $T^{\ast}: \mathcal{Y} \longrightarrow \mathcal{X}$ such that $\langle T \xi, \eta\rangle=\left\langle \xi, T^{\ast} \eta\right\rangle$ for all $\xi \in \mathcal{X}, \eta \in \mathcal{Y}$, and is called bounded if for all $p \in S(\mathcal{A})$, there is $M_{p}>0$ such that $\bar{p}_{\mathcal{Y}}(T \xi) \leq M_{p} \bar{p}_{\mathcal{X}}(\xi)$ for all $\xi \in \mathcal{X}$.
		
		We denote by $Hom_{\mathcal{A}}^{\ast}(\mathcal{X}, \mathcal{Y})$, the set of all adjointable operator from $\mathcal{X}$ to $\mathcal{Y}$ and $Hom_{\mathcal{A}}^{\ast}(\mathcal{X})=Hom_{\mathcal{A}}^{\ast}(\mathcal{X}, \mathcal{X})$	 
	\end{definition}
	
	\begin{definition}
		Let $\mathcal{A}$ be a pro-$C^{\ast}$-algebra and $\mathcal{X}, \mathcal{Y}$ be two Hilbert $\mathcal{A}$-modules. The operator $T: \mathcal{X} \rightarrow \mathcal{Y}$ is called uniformly bounded below, if there exists $C>0$ such that for each $p \in S(\mathcal{A})$,
		\begin{equation*}
			\bar{p}_{\mathcal{Y}}(T \xi) \leqslant C \bar{p}_{\mathcal{X}}(\xi), \quad \text { for all } \xi \in \mathcal{X} 
		\end{equation*}
		and  is called uniformly bounded  above if there exists $C^{\prime}>0$ such that for each $p \in S(\mathcal{A})$,
		\begin{equation*}
			\bar{p}_{\mathcal{Y}}(T \xi) \geqslant C^{\prime}  \bar{p}_{\mathcal{X}}(\xi), \quad \text { for all } \xi \in \mathcal{X}
		\end{equation*}
		
		\begin{equation*}
			\|T\|_{\infty}=\inf \{M: M \text { is an upper bound for } T\}
		\end{equation*}
		\begin{equation*}
			\hat{p}_{\mathcal{Y}}(T)=\sup \left\{\bar{p}_{\mathcal{Y}}(T(x)): \xi \in \mathcal{X}, \quad \bar{p}_{\mathcal{X}}(\xi) \leqslant 1\right\}
		\end{equation*}
		It's clear to see that,  $\hat{p}(T) \leqslant\|T\|_{\infty}$ for all $p \in S(\mathcal{A})$.	
	\end{definition}
	\begin{proposition}\cite{Azhini}. \label{Prop2.6}
		Let $\mathcal{X}$ be a Hilbert module over pro-$C^{*}$-algebra $\mathcal{A}$ and $T$ be an invertible element in $Hom_{\mathcal{A}}^{\ast}(\mathcal{X})$ such that both are uniformly bounded. Then for each $\xi \in \mathcal{X}$,
		$$
		\left\|T^{-1}\right\|_{\infty}^{-2}\langle \xi, \xi\rangle \leq\langle T \xi, T \xi \rangle \leq\|T\|_{\infty}^{2}\langle \xi, \xi\rangle
		.$$
	\end{proposition}
	
	Similar to $C^{\ast}$-algebra the $\ast$-homomorphism between two pro-$C^{\ast}$-algebra is increasing
	\begin{lemma}\label{5} 
		If $\varphi:\mathcal{A}\longrightarrow\mathcal{B}$ is an $\ast$-homomorphism between pro-$\mathcal{C}^{\ast}$-algebras, then $\varphi$ is increasing, that is, if $a\leq b$, then $\varphi(a)\leq\varphi(b)$.
	\end{lemma}
	
	\section{$\ast$-operator frame for $Hom_{\mathcal{A}}^{\ast}(\mathcal{X})$}
	\begin{definition}
		A family of adjointable operators $\{T_{i}\}_{i\in J}$ on a Hilbert $\mathcal{A}$-module $\mathcal{X}$ over a unital pro-$C^{\ast}$-algebra is said to be an operator frame for $Hom_{\mathcal{A}}^{\ast}(\mathcal{X})$, if there exists positive constants $A, B > 0$ such that 
		\begin{equation}\label{eq3}
			A\langle \xi,\xi\rangle\leq\sum_{i\in J}\langle T_{i}\xi,T_{i}\xi\rangle\leq B\langle \xi,\xi\rangle, \forall \xi\in\mathcal{X}.
		\end{equation}
		The numbers $A$ and $B$ are called lower and upper bounds of the operator frame, respectively. If $A=B=\lambda$, the operator frame is $\lambda$-tight. If $A = B = 1$, it is called a normalized tight operator frame or a Parseval operator frame.
		If only upper inequality of \eqref{eq3} hold, then $\{T_{i}\}_{i\in J}$ is called an operator Bessel sequence for $Hom_{\mathcal{A}}^{\ast}(\mathcal{X})$.
	\end{definition}
	
	\begin{definition}
		A family of adjointable operators $\{T_{i}\}_{i\in I}$ on a Hilbert $\mathcal{A}$-module $\mathcal{X}$ over a pro-$C^{\ast}$-algebra is said to be an $\ast$-operator frame for $Hom_{\mathcal{A}}^{\ast}(\mathcal{X})$, if there exists two strictly nonzero elements $A$ and $B$ in $\mathcal{A}$ such that 
		\begin{equation}\label{eqq33}
			A\langle \xi,\xi\rangle A^{\ast}\leq\sum_{i\in I}\langle T_{i}\xi,T_{i}\xi\rangle\leq B\langle \xi,\xi\rangle B^{\ast}, \forall \xi\in\mathcal{X}.
		\end{equation}
		The elements $A$ and $B$ are called lower and upper bounds of the $\ast$-operator frame, respectively. If $A=B=\lambda$, the $\ast$-operator frame is $\lambda$-tight. If $A = B = 1_{\mathcal{A}}$, it is called a normalized tight $\ast$-operator frame or a Parseval $\ast$-operator frame. If only upper inequality of \eqref{eqq33} hold, then $\{T_{i}\}_{i\in i}$ is called an $\ast$-operator Bessel sequence for $Hom_{\mathcal{A}}^{\ast}(\mathcal{X})$.
	\end{definition}
	We mentioned that the set of all of operator frames for $Hom_{\mathcal{A}}^{\ast}(\mathcal{X})$ can be considered
	as a subset of $\ast$-operator frame. To illustrate this, let $\{T_{j}\}_{i\in I}$ be an operator frame for Hilbert $\mathcal{A}$-module $\mathcal{X}$
	with operator frame real bounds $A$ and $B$. Note that for $\xi\in\mathcal{X}$,
	\begin{equation*}
		(\sqrt{A})1_{\mathcal{A}}\langle \xi,\xi\rangle(\sqrt{A})1_{\mathcal{A}}\leq\sum_{i\in I}\langle T_{i}\xi,T_{i}\xi\rangle\leq(\sqrt{B})1_{\mathcal{A}}\langle \xi,\xi\rangle(\sqrt{B})1_{\mathcal{A}}.
	\end{equation*}
	Therefore, every operator frame for $Hom_{\mathcal{A}}^{\ast}(\mathcal{X})$ with real bounds $A$ and $B$ is an $\ast$-operator frame for $Hom_{\mathcal{A}}^{\ast}(\mathcal{X})$ with $\mathcal{A}$-valued $\ast$-operator frame bounds $(\sqrt{A})1_{\mathcal{A}}$ and $(\sqrt{B})1_{\mathcal{B}}$.
	\begin{example}
		Let $\mathcal{A}$ be a Hilbert pro-$C^{\ast}$-module over itself with the inner product $\langle a,b\rangle=ab^{\ast}$.
		Let $\{\xi_{i}\}_{i\in I}$ be an $\ast$-frame for $\mathcal{A}$ with bounds $A$ and $B$, respectively. For each $i\in I$, we define $T_{i}:\mathcal{A}\to\mathcal{A}$ by $T_{i}\xi=\langle \xi,\xi_{i}\rangle,\;\; \forall \xi\in\mathcal{A}$. $T_{i}$ is adjointable and $T_{i}^{\ast}a=a\xi_{i}$ for each $a\in\mathcal{A}$. And we have 
		\begin{equation*}
			A\langle \xi,\xi\rangle A^{\ast}\leq\sum_{i\in I}\langle \xi,\xi_{i}\rangle\langle \xi_{i},\xi\rangle\leq B\langle \xi,\xi\rangle B^{\ast}, \forall \xi\in\mathcal{A}.
		\end{equation*}
		Then
		\begin{equation*}
			A\langle \xi,\xi\rangle A^{\ast}\leq\sum_{i\in I}\langle T_{i}\xi,T_{i}\xi\rangle\leq B\langle \xi,\xi\rangle B^{\ast}, \forall \xi\in\mathcal{A}.
		\end{equation*}
		So $\{T_{i}\}_{i\in I}$ is an $\ast$-operator frame in $\mathcal{A}$ with bounds $A$ and $B$, respectively.
	\end{example}
	Similar to $\ast$-frames, we introduce the $\ast$-operator frame transform and $\ast$-frame operator and establish some properties.
	
	\begin{theorem} \label{2.3}
		Let $\{T_{i}\}_{i\in I}\subset Hom_{\mathcal{A}}^{\ast}(\mathcal{X})$ be an $\ast$-operator frame with lower and upper bounds $A$ and $B$, respectively. The $\ast$-operator frame transform $R:\mathcal{X}\rightarrow l^{2}(\mathcal{X})$ defined by $R\xi=\{T_{i}\xi\}_{i\in I}$ is injective and closed range adjointable $\mathcal{A}$-module map and $\bar{p}_{\mathcal{X}}(R)\leq \bar{p}_{\mathcal{X}}(B)$. The adjoint operator $R^{\ast}$ is surjective and it is given by $R^{\ast}(\{\xi_{i}\}_{i\in I})=\sum_{i\in I}T_{i}^{\ast}\xi_{i}$ for all $\{\xi_{i}\}_{i\in I}$ in $l^{2}(\mathcal{X})$.
	\end{theorem}
	\begin{proof}
		By the definition of norm in $l^{2}(\mathcal{X})$
		\begin{equation}\label{3.3}
			\bar{p}_{\mathcal{X}}(R\xi)^{2}=p(\sum_{i\in I}\langle T_{i}\xi,T_{i}\xi\rangle)\leq \bar{p}_{\mathcal{X}}(B)^{2}p(\langle \xi,\xi\rangle), \forall \xi\in\mathcal{X}.
		\end{equation}
		This inequality implies that $R$ is well defined and $\bar{p}_{\mathcal{X}}(R)\leq \bar{p}_{\mathcal{X}}(B)$. Clearly, $R$ is a linear $\mathcal{A}$-module map. We now show that the range of $R$ is closed. Let $\{R\xi_{n}\}_{n\in\mathbb{N}}$ be a sequence in the range of $R$ such that $\lim_{n\to\infty}R\xi_{n}=\eta$. For $n, m\in\mathbb{N}$, we have
		\begin{equation*}
			p(A\langle \xi_{n}-\xi_{m},\xi_{n}-\xi_{m}\rangle A^{\ast})\leq p(\langle R(\xi_{n}-\xi_{m}),R(\xi_{n}-\xi_{m})\rangle)=\bar{p}_{\mathcal{X}}(R(\xi_{n}-\xi_{m}))^{2}.
		\end{equation*}
		Since $\{R\xi_{n}\}_{n\in\mathbb{N}}$ is Cauchy sequence in $\mathcal{X}$, then
		
		$p(A\langle \xi_{n}-\xi_{m},\xi_{n}-\xi_{m}\rangle A^{\ast})\rightarrow0$, as $n,m\rightarrow\infty.$
		
		Note that for $n, m\in\mathbb{N}$,
		\begin{align*}
			p(\langle \xi_{n}-\xi_{m},\xi_{n}-\xi_{m}\rangle)&=p(A^{-1}A\langle \xi_{n}-\xi_{m},\xi_{n}-\xi_{m}\rangle A^{\ast}(A^{\ast})^{-1})\\
			&\leq p(A^{-1})^{2} p(A\langle \xi_{n}-\xi_{m},\xi_{n}-\xi_{m}\rangle A^{\ast}).
		\end{align*}
		Therefore the sequence $\{\xi_{n}\}_{n\in\mathbb{N}}$ is Cauchy and hence there exists $\xi\in \mathcal{X}$ such that $\xi_{n}\rightarrow \xi$ as $n\rightarrow\infty$. Again by \eqref{3.3}, we have 
		$$\bar{p}_{\mathcal{X}}(R(\xi_{n}-\xi_{m}))^{2}\leq \bar{p}_{\mathcal{X}}(B)^{2} p(\langle \xi_{n}-\xi,\xi_{n}-\xi\rangle).$$
		
		Thus $p(R\xi_{n}-R\xi)\rightarrow0$ as $n\rightarrow\infty$ implies that $R\xi=\eta$. It concludes that the range of $R$ is closed. Next we show that $R$ is injective. Suppose that $\xi\in\mathcal{X}$ and $R\xi=0$. Note that $A\langle \xi,\xi\rangle A^{\ast}\leq\langle R\xi,R\xi\rangle$ then $\langle \xi,\xi\rangle=0$ so $\xi=0$ i.e. $R$ is injective.
		
		For $\xi\in\mathcal{X}$ and $\{\xi_{i}\}_{i\in I}\in l^{2}(\mathcal{X})$ we have $$\langle R\xi, \{\xi_{i}\}_{i\in I}\rangle=\langle \{T_{i}\xi\}_{i\in I}, \{\xi_{i}\}_{i\in I}\rangle=\sum_{i\in I}\langle T_{i}\xi, \xi_{i}\rangle=\sum_{i\in I}\langle \xi, T_{i}^{\ast}\xi_{i}\rangle=\langle \xi, \sum_{i\in I}T_{i}^{\ast}\xi_{i}\rangle.$$ 
		Then $R^{\ast}(\{\xi_{i}\}_{i\in I})=\sum_{i\in I}T_{i}^{\ast}\xi_{i}$. By injectivity of $R$, the operator $R^{\ast}$ has closed range and $\mathcal{X}=range(R^{\ast})$, which completes the proof.
	\end{proof}
	Now we define $\ast$-frame operator and studies some of its properties.
	\begin{definition}
		Let $\{T_{i}\}_{i\in I}\subset Hom_{\mathcal{A}}^{\ast}(\mathcal{X})$ be an $\ast$-operator frame with $\ast$-operator frame transform $R$ and lower and upper bounds $A$ and $B$, respectively. The $\ast$-frame operator $S:\mathcal{X}\to\mathcal{X}$ is defined by $S\xi=R^{\ast}R\xi=\sum_{i\in I}T_{i}^{\ast}T_{i}\xi,\;\;\forall \xi\in\mathcal{X}$.
	\end{definition}
	The following lemma is used to prove the next results
	\begin{lemma}\label{2.8} Let $\mathcal{X}$ and $\mathcal{Y}$ be two Hilbert $\mathcal{A}$-modules and $T\in Hom_{\mathcal{A}}^{\ast}(\mathcal{X},\mathcal{Y})$.
		\begin{itemize}
			\item [(i)] If $T$ is injective and $T$ has closed range, then the adjointable map $T^{\ast}T$ is invertible and $$\bar{p}_{\mathcal{X}}(T^{\ast}T^{-1})^{-1}I_{\mathcal{X}}\leq T^{\ast}T\leq \bar{p}_{\mathcal{X}}(T)^{2}I_{\mathcal{X}}.$$
			\item  [(ii)]	If $T$ is surjective, then the adjointable map $TT^{\ast}$ is invertible and $$\bar{p}_{\mathcal{X}}((TT^{\ast})^{-1})^{-1}I_{\mathcal{Y}}\leq TT^{\ast}\leq \bar{p}_{\mathcal{X}}(T)^{2}I_{\mathcal{Y}}.$$
		\end{itemize}	
	\end{lemma}
	\begin{proof}

		\begin{enumerate}
			\item\label{1}  Since the adjointable map $T^{*}$ is surjective, it follows that for any $\xi \in \mathcal{X}$, there exists $\eta \in \mathcal{Y}$ such that $T^{*} \eta=\xi$. Since $\mathcal{Y}= ker T^{*} \oplus \operatorname{Im} T$, it follows that $\eta=\eta_{1}+T h$, for some $\eta_{1} \in \operatorname{ker} T^{*}$ and some $h \in \mathcal{X}$. Thus, $\xi=T^{*}\left(\eta_{1}+T h\right)=T^{*} T h$, and hence $T^{*} T$ is surjective.	If $T^{*} T \xi=0$, then $T \xi \in \operatorname{ker} T^{*} \cap \operatorname{Im} T=$ $\{0\}$, which implies that $\xi=0$. therefore, $T^{\ast} T$ is an injective positive map. Hence, $T^{\ast} T$ is an invertible element of the set of all bounded $\mathcal{A}$-module maps, $0 \leq\left(T^{*} T\right)^{-1} \leq \bar{p}_{\mathcal{X}}(\left(T^{*} T\right)^{-1})$ and $0 \leq\left(T^{*} T\right) \leq \bar{p}_{\mathcal{X}}(\left(T^{*} T\right)) .$ Therefore $ \bar{p}_{\mathcal{X}}((T^{*} T)^{-1})^{-1} \leq T^{*} T \leq \bar{p}_{\mathcal{X}}(T)^{2}$
			\item Let $T$ be surjective. Then, $T^{\ast}$ is injective and $T^{\ast}$ has a closed range. By substituting $T^{\ast}$ for $T$ in (\ref{1}), we have $T T^{*}$ is invertible and $\bar{p}_{\mathcal{X}}(\left(T T^{*}\right)^{-1})^{-1} \leq T T^{*} \leq \bar{p}_{\mathcal{X}}(T)^{2}$
		\end{enumerate}
	\end{proof}
	\begin{theorem}
		The $\ast$-operator frame $S$ is bounded, positive, self-adjoint, invertible and $\bar{p}_{\mathcal{X}}(A^{-1})^{-2}\leq\bar{p}_{\mathcal{X}}(S)\leq \bar{p}_{\mathcal{X}}(B)^{2}$.
	\end{theorem}
	\begin{proof}
		
		By definition we have, $\forall \xi, \eta \in\mathcal{X}$:
		\begin{align*}
			\langle S\xi,\eta\rangle&=\left\langle\sum_{i\in I}T_{i}^{\ast}T_{i}\xi,\eta\right\rangle\\
			&=\sum_{i\in I}\langle T_{i}^{\ast}T_{i}\xi,\eta\rangle\\
			&=\sum_{i\in I}\langle \xi,T_{i}^{\ast}T_{i}\eta\rangle \\
			&=\left\langle \xi,\sum_{i\in I}T_{i}^{\ast}T_{i}\eta\right\rangle\\
			&=\langle \xi,S\eta\rangle.
		\end{align*}
		Then $S$ is a selfadjoint.
		
		By Lemma \ref{2.8} and Theorem \ref{2.3}, $S$ is invertible. Clearly $S$ is positive.
		
		By definition of an $\ast$-operator frame we have
		\begin{equation*}
			A\langle \xi,\xi\rangle A^{\ast}\leq\sum_{i\in I}\langle T_{i}\xi,T_{i}\xi\rangle\leq B\langle \xi,\xi\rangle B^{\ast}.
		\end{equation*}
		So
		\begin{equation*}
			A\langle \xi,\xi\rangle A^{\ast}\leq\langle S\xi,\xi\rangle\leq B\langle \xi,\xi\rangle B^{\ast}.
		\end{equation*}
		This give
		\begin{equation*}
			\bar{p}_{\mathcal{X}}(A^{-1})^{-2}\bar{p}_{\mathcal{X}}(\xi)^{2}\leq\bar{p}_{\mathcal{X}}(\langle S\xi,\xi\rangle)\leq \bar{p}_{\mathcal{X}}(B)^{2}\bar{p}_{\mathcal{X}}(\xi)^{2}, \forall \xi\in\mathcal{X}.
		\end{equation*}
		If we take supremum on all $\xi\in\mathcal{X}$, where $\bar{p}_{\mathcal{X}}(\xi)\leq1$, then $\bar{p}_{\mathcal{X}}(A^{-1})^{-2}\leq \bar{p}_{\mathcal{X}}(S)\leq \bar{p}_{\mathcal{X}}(B)^{2}$.
		
	\end{proof}
	\begin{corollary}
		Let $\{T_{i}\}_{i\in I}\subset Hom_{\mathcal{A}}^{\ast}(\mathcal{X})$ be an $\ast$-operator frame with $\ast$-operator frame transform $R$ and lower and upper bounds $A$ and $B$, respectively. Then $\{T_{i}\}_{i\in I}$ is an operator frame for $\mathcal{X}$ with lower and upper bounds $\bar{p}_{\mathcal{X}}((R^{\ast}R)^{-1})^{-1}$ and $\bar{p}_{\mathcal{X}}(R)^{2}$, respectively.
	\end{corollary}
	\begin{proof}
		By Theorem \ref{2.3}, $R$ is injective and has closed range and by Lemma \ref{2.8}
		$$\bar{p}_{\mathcal{X}}((R^{\ast}R)^{-1})^{-1}I_{\mathcal{X}}\leq R^{\ast}R\leq \bar{p}_{\mathcal{X}}(R)^{2}I_{\mathcal{X}}.$$
		So$$\bar{p}_{\mathcal{X}}((R^{\ast}R)^{-1})^{-1}\langle \xi, \xi\rangle\leq \sum_{i\in I}\langle T_{i}\xi, T_{i}\xi\rangle\leq\bar{p}_{\mathcal{X}}(R)^{2}\langle \xi, \xi\rangle,\;\;\forall x\in\mathcal{X}.$$
		Then $\{T_{i}\}_{i\in I}$ is an operator frame for $\mathcal{X}$ with lower and upper bounds $\bar{p}_{\mathcal{X}}((R^{\ast}R)^{-1})^{-1}$ and $\bar{p}_{\mathcal{X}}(R)^{2}$, respectively.
	\end{proof}	
	
	\begin{theorem} \label{Th3.9}
		Let $\{T_{i}\}_{i\in I}\subset Hom_{\mathcal{A}}^{\ast}(\mathcal{X})$ be an $\ast$-operator frame for $\mathcal{X}$, with lower and upper bounds $A$ and $B$, respectively and with $\ast$-frame operator $S$. Let $\theta\in Hom_{\mathcal{A}}^{\ast}(\mathcal{X})$ be injective and has a closed range. Then $\{T_{i}\theta\}_{i\in I}$ is an $\ast$-operator frame for $\mathcal{X}$ with $\ast$-frame operator $\theta^{\ast}S\theta$ with bounds $\bar{p}_{\mathcal{X}}((\theta^{\ast}\theta)^{-1})^{-\frac{1}{2}}A$, $\bar{p}_{\mathcal{X}}(\theta) B$.
	\end{theorem}
	\begin{proof}
		We have 
		\begin{equation}\label{eq11}
			A\langle\theta \xi, \theta \xi\rangle A^{\ast}\leq\sum_{i\in I}\langle T_{i}\theta \xi, T_{i}\theta \xi\rangle\leq B\langle\theta \xi, \theta \xi\rangle B^{\ast}, \forall \xi\in\mathcal{X}.
		\end{equation}
		Using Lemma \ref{2.8}, we have $\bar{p}_{\mathcal{X}}((\theta^{\ast}\theta)^{-1})^{-1}\langle \xi,\xi\rangle\leq\langle \theta \xi,\theta \xi\rangle$, $\forall \xi\in\mathcal{X}$. This implies
		\begin{equation}\label{eq22} 
			\bar{p}_{\mathcal{X}}((\theta^{\ast}\theta)^{-1})^{-\frac{1}{2}}A\langle \xi,\xi\rangle(\bar{p}_{\mathcal{X}}((\theta^{\ast}\theta)^{-1})^{-\frac{1}{2}}A)^{\ast}\leq A\langle \theta \xi,\theta \xi\rangle A^{\ast}, \forall \xi\in\mathcal{X}.
		\end{equation}
		And we know that $\langle \theta \xi,\theta \xi\rangle\leq \bar{p}_{\mathcal{X}}(\theta)^{2}\langle \xi,\xi\rangle$, $\forall \xi\in\mathcal{X}$. This implies that
		\begin{equation}\label{eq33}
			B\langle \theta \xi,\theta \xi\rangle B^{\ast}\leq\bar{p}_{\mathcal{X}}(\theta)B\langle \xi,\xi\rangle(\bar{p}_{\mathcal{X}}(\theta)B)^{\ast}, \forall \xi\in\mathcal{X}.
		\end{equation}
		Using \eqref{eq11}, \eqref{eq22}, \eqref{eq33} we have
		\begin{align*}
			\bar{p}_{\mathcal{X}}((\theta^{\ast}\theta)^{-1})^{-\frac{1}{2}}A\langle \xi,\xi\rangle(\bar{p}_{\mathcal{X}}((\theta^{\ast}\theta)^{-1})^{-\frac{1}{2}}A)^{\ast}&\leq\sum_{i\in I}\langle T_{i}\theta x, T_{i}\theta \xi\rangle\\&\leq\bar{p}_{\mathcal{X}}(\theta) B\langle \xi,\xi\rangle(\bar{p}_{\mathcal{X}}(\theta )B)^{\ast}, \forall \xi\in\mathcal{X}.
		\end{align*}
		So $\{T_{i}\theta\}_{i\in I}$ is an $\ast$-operator frame for $\mathcal{X}$.
		
		Moreover for every $\xi\in\mathcal{X}$, we have
		$$\theta^{\ast}S\theta \xi=\theta^{\ast}\sum_{i\in I}T_{i}^{\ast}T_{i}\theta \xi=\sum_{i\in I}\theta^{\ast}T_{i}^{\ast}T_{i}\theta \xi=\sum_{i\in I}(T_{i}\theta)^{\ast}(T_{i}\theta)\xi.$$ This completes the proof.	
	\end{proof}
	\begin{corollary}
		Let $\{T_{i}\}_{i\in I}\subset Hom_{\mathcal{A}}^{\ast}(\mathcal{X})$ be an $\ast$-operator frame for $\mathcal{X}$, with $\ast$-frame operator $S$. Then $\{T_{i}S^{-1}\}_{i\in I}$ is an $\ast$-operator frame for $\mathcal{X}$.
	\end{corollary}
	\begin{proof}
		Result from theorem \ref{Th3.9} by taking $\theta=S^{-1}$.
	\end{proof}	
	\begin{corollary}
		Let $\{T_{i}\}_{i\in I}\subset Hom_{\mathcal{A}}^{\ast}(\mathcal{X})$ be an $\ast$-operator frame for $\mathcal{X}$, with $\ast$-frame operator $S$. Then $\{T_{i}S^{-\frac{1}{2}}\}_{i\in I}$ is a Parseval $\ast$-operator frame for $\mathcal{X}$.
	\end{corollary}
	\begin{proof}
		Result from theorem \ref{Th3.9} by taking $\theta=S^{-\frac{1}{2}}$.
	\end{proof}
	
	\begin{theorem} \label{.}
		Let $\{T_{i}\}_{i\in I}\subset Hom_{\mathcal{A}}^{\ast}(\mathcal{X})$ be an $\ast$-operator frame for $\mathcal{X}$, with lower and upper bounds $A$ and $B$, respectively. Let $\theta\in Hom_{\mathcal{A}}^{\ast}(\mathcal{X})$ be surjective. Then $\{\theta T_{i}\}_{i\in I}$ is an $\ast$-operator frame for $\mathcal{X}$ with bounds $A\bar{p}_{\mathcal{X}}((\theta\theta^{\ast})^{-1})^{-\frac{1}{2}}$, $B\bar{p}_{\mathcal{X}}(\theta)$.
	\end{theorem}
	\begin{proof}
		By the definition of $\ast$-operator frame, we have
		\begin{equation}\label{eqq11}
			A\langle \xi, \xi\rangle A^{\ast}\leq\sum_{i\in I}\langle T_{i}\xi, T_{i}\xi\rangle\leq B\langle \xi, \xi\rangle B^{\ast}, \forall \xi\in\mathcal{X}.
		\end{equation}
		Using Lemma \ref{2.8}, we have
		\begin{equation}\label{eqq111} \bar{p}_{\mathcal{X}}((\theta\theta^{\ast})^{-1})^{-1}\langle T_{i}\xi, T_{i}\xi\rangle\leq\langle \theta T_{i}x,\theta T_{i}\xi\rangle\leq\bar{p}_{\mathcal{X}}^{2}\langle T_{i}\xi, T_{i}\xi\rangle, \forall \xi\in\mathcal{X}.
		\end{equation}
		Using \eqref{eqq11}, \eqref{eqq111}, we have
		\begin{align*}
			\bar{p}_{\mathcal{X}}((\theta\theta^{\ast})^{-1})^{-\frac{1}{2}}A\langle \xi,\xi\rangle(\bar{p}_{\mathcal{X}}((\theta\theta^{\ast})^{-1})^{-\frac{1}{2}}A)^{\ast}&\leq\sum_{i\in I}\langle\theta T_{i}\xi, \theta T_{i}\xi\rangle\\&\leq B\bar{p}_{\mathcal{X}}(\theta)\langle \xi,\xi\rangle(B\bar{p}_{\mathcal{X}}(\theta))^{\ast}, \forall \xi\in\mathcal{X}.
		\end{align*}
		So $\{\theta T_{i}\}_{i\in I}$ is an $\ast$-operator frame for $\mathcal{X}$.
	\end{proof}
	\begin{theorem}
		Let $(\mathcal{X},\mathcal{A},\langle.,.\rangle_{\mathcal{A}})$ and $(\mathcal{X},\mathcal{B},\langle.,.\rangle_{\mathcal{B}})$ be two Hilbert pro-$\mathcal{C^{\ast}}$-modules and let $\varphi :\mathcal{A}\longrightarrow \mathcal{B}$ be a $\ast$-homomorphism and $\theta$ be a map on $\mathcal{X}$ such that $\langle \theta \xi,\theta \eta\rangle_{\mathcal{B}}=\varphi(\langle \xi, \eta\rangle_{\mathcal{A}})$ for all $\xi,\eta\in\mathcal{X}$. Also, suppose that $\{T_{i}\}_{i\in I}\subset Hom_{\mathcal{A}}^{\ast}(\mathcal{X})$ is an $\ast$-operator frame for $(\mathcal{X},\mathcal{A},\langle.,.\rangle_{\mathcal{A}})$ with $\ast$-frame operator $S_{\mathcal{A}} $ and lower and upper $\ast$-operator frame bounds $A$, $B$  respectively. If $\theta$ is surjective and $\theta T_{i}= T_{i}\theta$ for each $i$ in $I$, then $\{T_{i}\}_{i\in I}$ is an $\ast$-operator frame for $(\mathcal{X},\mathcal{B},\langle.,.\rangle_{\mathcal{B}})$ with $\ast$-frame operator $S_{\mathcal{B}} $ and lower and upper $\ast$-operator frame bounds $\varphi(A)$, $\varphi(B)$ respectively, and $\langle S_{\mathcal{B}}\theta \xi,\theta \eta\rangle_{\mathcal{B}}=\varphi(\langle S_{\mathcal{A}}\xi, \eta\rangle_{\mathcal{A}})$.
	\end{theorem}
	\begin{proof} Let $\eta\in\mathcal{X}$ then there exists $\xi\in\mathcal{X}$ such that $\theta \xi=\eta$ ($\theta$ is surjective). By the definition of $\ast$-operator frames we have
		$$A\langle \xi,\xi\rangle_{\mathcal{A}} A^{\ast}\leq\sum_{i\in I}\langle T_{i}\xi, T_{i}\xi\rangle_{\mathcal{A}}\leq B\langle \xi,\xi\rangle_{\mathcal{A}} B^{\ast}.$$
		By lemma \ref{5} we have
		$$\varphi(A\langle \xi,\xi\rangle_{\mathcal{A}} A^{\ast})\leq\varphi(\sum_{i\in I}\langle T_{i}\xi, T_{i}\xi\rangle_{\mathcal{A}})\leq\varphi( B\langle \xi,\xi\rangle_{\mathcal{A}} B^{\ast}).$$
		By the definition of $\ast$-homomorphism we have
		$$\varphi(A)\varphi(\langle \xi,\xi\rangle_{\mathcal{A}}) \varphi(A^{\ast})\leq\sum_{i\in I}\varphi(\langle T_{i}\xi, T_{i}\xi\rangle_{\mathcal{A}})\leq\varphi( B)\varphi(\langle \xi,\xi\rangle_{\mathcal{A}}) \varphi(B^{\ast}).$$
		By the relation betwen $\theta$ and $\varphi$ we get
		$$\varphi(A)\langle \theta \xi,\theta \xi\rangle_{\mathcal{B}} \varphi(A)^{\ast}\leq\sum_{i\in I}\langle \theta T_{i}\xi,\theta T_{i}\xi\rangle_{\mathcal{B}}\leq\varphi( B)\langle\theta \xi,\theta \xi\rangle_{\mathcal{B}} \varphi(B)^{\ast}.$$
		By the relation betwen $\theta$ and $T_{i}$ we have
		$$\varphi(A)\langle \theta \xi,\theta \xi\rangle_{\mathcal{B}} \varphi(A)^{\ast}\leq\sum_{i\in I}\langle T_{i}\theta \xi, T_{i}\theta \xi\rangle_{\mathcal{B}}\leq\varphi( B)\langle\theta \xi,\theta \xi\rangle_{\mathcal{B}} \varphi(B)^{\ast}.$$
		Then
		$$
		\varphi(A)\langle  \eta, \eta \rangle_{\mathcal{B}} (\varphi(A))^{\ast}\leq\sum_{i\in I}\langle T_{i}\eta, T_{i}\eta\rangle_{\mathcal{B}}
		\leq\varphi( B)\langle \eta,\eta\rangle_{\mathcal{B}} (\varphi(B))^{\ast} , \forall \eta\in\mathcal{X}.
		$$
		On the other hand we have
		$$
		\aligned
		\varphi(\langle S_{\mathcal{A}}\xi, \eta\rangle_{\mathcal{A}})&=\varphi(\langle\sum_{i\in I}T_{i}^{\ast}T_{i}\xi,\eta\rangle_{\mathcal{A}})\\
		&=\sum_{i\in I}\varphi(\langle T_{i}\xi, T_{i}\eta\rangle_{\mathcal{A}})\\
		&=\sum_{i\in I}\langle\theta T_{i}\xi,\theta T_{i}\eta\rangle_{\mathcal{B}}
		\\
		&=\sum_{i\in I}\langle T_{i}\theta \xi, T_{i}\theta \eta\rangle_{\mathcal{B}}\\
		&=\langle\sum_{i\in I}T_{i}^{\ast}T_{i}\theta \xi,\theta \eta\rangle_{\mathcal{B}}\\
		&=\langle S_{\mathcal{B}}\theta \xi,\theta \eta\rangle_{\mathcal{B}}.
		\endaligned
		$$
		Which completes the proof.
	\end{proof}
	\section{Tensor product}
	The minimal or injective tensor product of the pro-$C^{\ast}$-algebras $\mathcal{A}$ and $\mathcal{B}$, denoted by $\mathcal{A} \otimes \mathcal{B}$, is the completion of the algebraic tensor product $\mathcal{A} \otimes_{\text {alg }} \mathcal{B}$ with respect to the topology determined by a family of $C^{\ast}$-seminorms. Suppose that $\mathcal{X}$ is a Hilbert module over a pro-$C^{\ast}$-algebra $\mathcal{A}$ and $\mathcal{Y}$ is a Hilbert module over a pro-$C^{\ast}$-algebra $\mathcal{B}$. The algebraic tensor product $\mathcal{X} \otimes_{\text {alg }} \mathcal{Y}$ of $\mathcal{X}$ and $\mathcal{Y}$ is a pre-Hilbert $ \mathcal{A} \otimes \mathcal{B}$-module with the action of $ \mathcal{A} \otimes \mathcal{B}$ on $ \mathcal{X} \otimes_{\text {alg }}\mathcal{Y}$ defined by
	$$
	(\xi \otimes \eta)(a \otimes b)=\xi a \otimes \eta b
	\; \; \text{for all} \; \xi \in \mathcal{X} ,\eta \in \mathcal{Y}, a \in \mathcal{A} \; \text{and} \; b \in \mathcal{B}$$
	and the inner product 
	$$
	\langle\cdot, \cdot\rangle:\left(\mathcal{X} \otimes_{\text {alg }} \mathcal{Y}\right) \times\left(\mathcal{X} \otimes_{\text {alg }} \mathcal{Y}\right) \rightarrow \mathcal{A} \otimes_{\text {alg }}\mathcal{B}. \text { defined by }
	$$
	$$
	\left\langle\xi_{1} \otimes \eta_{1}, \xi_{2} \otimes \eta_{2}\right\rangle=\left\langle\xi_{1}, \xi_{2}\right\rangle \otimes\left\langle\eta_{1}, \eta_{2}\right\rangle
	$$
	We also know that for $ z=\sum_{i=1}^{n}\xi_{i}\otimes \eta_{i} $ in $\mathcal{X}\otimes_{alg}\mathcal{Y}$ we have $ \langle z,z\rangle_{\mathcal{A}\otimes\mathcal{B}}=\sum_{i,j}\langle \xi_{i},\xi_{j}\rangle_{\mathcal{A}}\otimes\langle \eta_{i},\eta_{j}\rangle_{\mathcal{B}}\geq0 $ and $ \langle z,z\rangle_{\mathcal{A}\otimes\mathcal{B}}=0 $ iff $z=0$.\\
	The external tensor product of $\mathcal{X}$ and $\mathcal{Y}$ is the Hilbert module $\mathcal{X} \otimes \mathcal{Y}$ over $\mathcal{A} \otimes \mathcal{B}$ obtained by the completion of the pre-Hilbert $\mathcal{A} \otimes \mathcal{B}$-module $\mathcal{X} \otimes_{\text {alg }} \mathcal{Y}$.
	
	If $P \in M(\mathcal{X})$ and $Q \in M(\mathcal{Y})$ then there is a unique adjointable module morphism $P \otimes Q: \mathcal{A} \otimes  \mathcal{B} \rightarrow \mathcal{X} \otimes \mathcal{Y}$ such that $(P \otimes Q)(a \otimes b)=P(a) \otimes Q(b)$ and $(P \otimes Q)^{*}(a \otimes b)=P^{*}(a) \otimes Q^{*}(b)$ for all $a \in A$ and for all $b \in B$ (see, for example, cite The minimal or injective tensor product of the pro-$C^{\ast}$-algebras $\mathcal{A}$ and $\mathcal{B}$, denoted by $\mathcal{A} \otimes \mathcal{B}$, is the completion of the algebraic tensor product $\mathcal{A} \otimes_{\text {alg }} \mathcal{B}$ with respect to the topology determined by a family of $C^{\ast}$-seminorms. Suppose that $\mathcal{X}$ is a Hilbert module over a pro-$C^{\ast}$-algebra $\mathcal{A}$ and $\mathcal{Y}$ is a Hilbert module over a pro-$C^{\ast}$-algebra $\mathcal{B}$. The algebraic tensor product $\mathcal{X} \otimes_{\text {alg }} \mathcal{Y}$ of $\mathcal{X}$ and $\mathcal{Y}$ is a pre-Hilbert $ \mathcal{A} \otimes \mathcal{B}$-module with the action of $ \mathcal{A} \otimes \mathcal{B}$ on $ \mathcal{X} \otimes_{\text {alg }}\mathcal{Y}$ defined by
	$$
	(\xi \otimes \eta)(a \otimes b)=\xi a \otimes \eta b
	\; \; \text{for all} \; \xi \in \mathcal{X} ,\eta \in \mathcal{Y}, a \in \mathcal{A} \; \text{and} \; b \in \mathcal{B}$$
	and the inner product 
	$$
	\langle\cdot, \cdot\rangle:\left(\mathcal{X} \otimes_{\text {alg }} \mathcal{Y}\right) \times\left(\mathcal{X} \otimes_{\text {alg }} \mathcal{Y}\right) \rightarrow \mathcal{A} \otimes_{\text {alg }}\mathcal{B}. \text { defined by }
	$$
	$$
	\left\langle\xi_{1} \otimes \eta_{1}, \xi_{2} \otimes \eta_{2}\right\rangle=\left\langle\xi_{1}, \xi_{2}\right\rangle \otimes\left\langle\eta_{1}, \eta_{2}\right\rangle
	$$
	We also know that for $ z=\sum_{i=1}^{n}\xi_{i}\otimes \eta_{i} $ in $\mathcal{X}\otimes_{alg}\mathcal{Y}$ we have $ \langle z,z\rangle_{\mathcal{A}\otimes\mathcal{B}}=\sum_{i,j}\langle \xi_{i},\xi_{j}\rangle_{\mathcal{A}}\otimes\langle \eta_{i},\eta_{j}\rangle_{\mathcal{B}}\geq0 $ and $ \langle z,z\rangle_{\mathcal{A}\otimes\mathcal{B}}=0 $ iff $z=0$.\\
	The external tensor product of $\mathcal{X}$ and $\mathcal{Y}$ is the Hilbert module $\mathcal{X} \otimes \mathcal{Y}$ over $\mathcal{A} \otimes \mathcal{B}$ obtained by the completion of the pre-Hilbert $\mathcal{A} \otimes \mathcal{B}$-module $\mathcal{X} \otimes_{\text {alg }} \mathcal{Y}$.
	
	If $P \in M(\mathcal{X})$ and $Q \in M(\mathcal{Y})$ then there is a unique adjointable module morphism $P \otimes Q: \mathcal{A} \otimes  \mathcal{B} \rightarrow \mathcal{X} \otimes \mathcal{Y}$ such that $(P \otimes Q)(a \otimes b)=P(a) \otimes Q(b)$ and $(P \otimes Q)^{*}(a \otimes b)=P^{*}(a) \otimes Q^{*}(b)$ for all $a \in A$ and for all $b \in B$ (see, for example, \cite{Joita})
	Let I and J be countable index sets.
	\begin{theorem}
		Let $\mathcal{X}$ and $\mathcal{Y}$ be two Hilbert pro-$C^{\ast}$-modules over  pro-$C^{\ast}$-algebras $\mathcal{A}$ and $\mathcal{B}$, respectively. Let $\{T_{i}\}_{i\in I}\subset Hom_{\mathcal{A}}^{\ast}(\mathcal{X})$ and $\{L_{j}\}_{j\in J}\subset Hom_{\mathcal{B}}^{\ast}(\mathcal{Y})$ be two $\ast$-operator frames for $\mathcal{X}$ and $\mathcal{Y}$ with $\ast$-frame operators $S_{T}$ and $S_{L}$ and $\ast$-operator frame bounds $(A,B)$ and $(C,D)$ respectively. Then $\{T_{i}\otimes L_{j}\}_{i\in I,j\in J}  $ is an $\ast$-operator frame for Hibert $\mathcal{A}\otimes\mathcal{B}$-module $\mathcal{X}\otimes\mathcal{Y}$ with $\ast$-frame operator $ S_{T}\otimes S_{L}$ and lower and upper $\ast$-operator frame bounds $A\otimes C$ and $ B\otimes D $, respectively.
	\end{theorem}
	\begin{proof}
		By the definition of $\ast$-operator frames $\{T_{i}\}_{i\in I} $ and $\{L_{j}\}_{j\in J}$ we have 
		$$A\langle \xi,\xi\rangle_{\mathcal{A}} A^{\ast}\leq\sum_{i\in I}\langle T_{i}\xi,T_{i}\xi\rangle_{\mathcal{A}}\leq B\langle \xi,\xi\rangle_{\mathcal{A}} B^{\ast} , \forall \xi\in\mathcal{X},$$
		and
		$$C\langle \eta,\eta\rangle_{\mathcal{B}} C^{\ast}\leq\sum_{j\in J}\langle L_{j}\eta,L_{j}\eta\rangle_{\mathcal{B}}\leq D\langle \eta,\eta\rangle_{\mathcal{B}} D^{\ast} , \forall \eta\in\mathcal{Y}.$$
		Therefore
		$$
		\aligned
		&(A\langle \xi,\xi\rangle_{\mathcal{A}} A^{\ast})\otimes (C\langle \eta,\eta\rangle_{\mathcal{B}} C^{\ast})\\&\leq\sum_{i\in I}\langle T_{i}\xi,T_{i}\xi\rangle_{\mathcal{A}}\otimes\sum_{j\in J}\langle L_{j}\eta,L_{j}\eta\rangle_{\mathcal{B}}\\
		&\leq (B\langle \xi,\xi\rangle_{\mathcal{A}} B^{\ast})\otimes (D\langle \eta,\eta\rangle_{\mathcal{B}} D^{\ast}) , \forall \xi\in\mathcal{X} ,\forall \eta\in\mathcal{Y}.
		\endaligned
		$$
		Then
		$$
		\aligned
		&(A\otimes C)(\langle \xi,\xi\rangle_{\mathcal{A}}\otimes\langle \eta,\eta\rangle_{\mathcal{B}}) (A^{\ast}\otimes C^{\ast})\\&\leq\sum_{i\in I,j\in J}\langle T_{i}\xi,T_{i}\xi\rangle_{\mathcal{A}}\otimes\langle L_{j}\eta,L_{j}\eta\rangle_{\mathcal{B}}\\
		&\leq (B\otimes D)(\langle \xi,\xi\rangle_{\mathcal{A}}\otimes\langle \eta,\eta\rangle_{\mathcal{B}}) (B^{\ast}\otimes D^{\ast}) , \forall \xi\in\mathcal{X} ,\forall \eta\in\mathcal{Y}.
		\endaligned
		$$
		Consequently we have
		$$
		\aligned
		&(A\otimes C)\langle \xi\otimes \eta,\xi\otimes \eta\rangle_{\mathcal{A\otimes B}} (A\otimes C)^{\ast}\\&\leq\sum_{i\in I,j\in J}\langle T_{i}\xi\otimes L_{j}\eta,T_{i}\xi\otimes L_{j}\eta\rangle_{\mathcal{A\otimes B}} \\
		&\leq (B\otimes D)\langle \xi\otimes \eta,\xi\otimes \eta\rangle_{\mathcal{A\otimes B}} (B\otimes D)^{\ast}, \forall \xi\in\mathcal{X}, \forall \eta\in\mathcal{Y}.
		\endaligned
		$$
		Then for all $\xi\otimes \eta\in\mathcal{X\otimes Y}$ we have
		$$
		\aligned
		&(A\otimes C)\langle \xi\otimes \eta,\xi\otimes \eta\rangle_{\mathcal{A\otimes B}} (A\otimes C)^{\ast}\\&\leq\sum_{i\in I,j\in J}\langle(T_{i}\otimes L_{j})(\xi\otimes \eta),(T_{i}\otimes L_{j})(\xi\otimes \eta)\rangle_{\mathcal{A\otimes B}} \\
		&\leq (B\otimes D)\langle \xi\otimes \eta,\xi\otimes \eta\rangle_{\mathcal{A\otimes B}} (B\otimes D)^{\ast}.
		\endaligned
		$$
		The last inequality is satisfied for every finite sum of elements in $\mathcal{X}\otimes_{alg}\mathcal{Y}$ and then it's satisfied for all $z\in\mathcal{X\otimes Y}$. It shows that $\{T_{i}\otimes L_{j}\}_{i\in I,j\in J}  $ is $\ast$-operator frame for Hibert $\mathcal{A}\otimes\mathcal{B}$-module $\mathcal{X}\otimes\mathcal{Y}$ with lower and upper $\ast$-operator frame bounds $A\otimes C$ and $ B\otimes D $, respectively.\\
		By the definition of $\ast$-frame operator $S_{T}$ and $S_{L}$ we have:$$S_{T}\xi=\sum_{i\in I}T_{i}^{\ast}T_{i}\xi, \forall \xi\in\mathcal{X},$$
		and
		$$S_{L}\eta=\sum_{j\in J}L_{j}^{\ast}L_{j}\eta, \forall \eta\in\mathcal{Y}.$$
		Therefore
		$$
		\aligned
		(S_{T}\otimes S_{L})(\xi\otimes \eta)&=S_{T}\xi\otimes S_{L}\eta\\
		&=\sum_{i\in I} T_{i}^{\ast} T_{i}\xi\otimes\sum_{j\in J} L_{j}^{\ast} L_{j}\eta\\
		&=\sum_{i\in I,j\in J} T_{i}^{\ast} T_{i}\xi \otimes L_{j}^{\ast} L_{j}\eta\\
		&=\sum_{i\in I,j\in J}(T_{i}^{\ast}\otimes L_{j}^{\ast})(T_{i}\xi\otimes L_{j}\eta)\\
		&=\sum_{i\in I,j\in J}(T_{i}^{\ast}\otimes L_{j}^{\ast})(T_{i}\otimes L_{j})(\xi\otimes \eta)\\
		&=\sum_{i\in I,j\in J}( T_{i}\otimes L_{j})^{\ast})(L_{i}\otimes L_{j})(\xi\otimes \eta).
		\endaligned
		$$
		Now by the uniqueness of $\ast$-frame operator, the last expression is equal to $S_{T\otimes L}(\xi\otimes \eta)$. Consequently we have $ (S_{T}\otimes S_{L})(\xi\otimes \eta)=S_{T\otimes L}(\xi\otimes \eta)$. The last equality is satisfied for every finite sum of elements in $\mathcal{X}\otimes_{alg}\mathcal{Y}$ and then it's satisfied for all $z\in\mathcal{X\otimes Y}$. It shows that $ (S_{T}\otimes S_{L})(z)=S_{T\otimes L}(z)$. So $S_{T\otimes L}=S_{T}\otimes S_{L}$.
	\end{proof}

\end{document}